%% file: main.tex
\documentclass[lettersize,journal]{IEEEtran}
\usepackage{amsmath,amsfonts}
\usepackage{algorithmic}

\usepackage{algorithm}
\usepackage{array}
\usepackage[caption=false,font=normalsize,labelfont=sf,textfont=sf]{subfig}
\usepackage{textcomp}
\usepackage{stfloats}
\usepackage{url}
\usepackage{verbatim}
\usepackage{graphicx}
\usepackage{paralist}
\usepackage{cite}
\usepackage{amssymb}
\usepackage{enumerate}
\usepackage[numbers,sort&compress]{natbib}

\usepackage[figuresright]{rotating}
\usepackage{pdflscape}
\usepackage{amssymb}
\usepackage{latexsym}
\usepackage{longtable}
\usepackage{amsthm}
\usepackage{rotating,multirow,array}
\usepackage{enumerate}
\usepackage{microtype}
\usepackage{xcolor}
\usepackage[T1]{fontenc}    % use 8-bit T1 fonts

\usepackage{booktabs}       % professional-quality tables
\usepackage{nicefrac}       % compact symbols for 1/2, etc.
\usepackage{microtype}      % microtypography
\usepackage{xcolor}         % colors
\usepackage{subfiles}

\usepackage{float}

\theoremstyle{plain}
\newtheorem{theorem}{Theorem}[section]

\theoremstyle{definition}
\newtheorem{definition}[theorem]{Definition}
\newtheorem{assumption}[theorem]{Assumption}
\theoremstyle{remark}
\newtheorem{remark}[theorem]{Remark}

\definecolor{red}{rgb}{0.9,0,0}

\definecolor{blue}{rgb}{0,0,0.9}

\begin{document}
\title{A Distributed Semismooth Newton Based Augmented Lagrangian Method for Distributed Optimization}
\author{Qihao Ma\thanks{Qihao Ma is with School of Mathematics, Southwest Jiaotong University, No.999, Xian Road, West Park, High-tech Zone, Chengdu 611756, China ({3387570049@qq.com}).}, Chengjing Wang\thanks{ Chengjing Wang ({\bf corresponding author}) is with School of Mathematics, Southwest Jiaotong University, No.999, Xian Road, West Park, High-tech Zone, Chengdu 611756, China ({renascencewang@hotmail.com}).}, Peipei Tang\thanks{Peipei Tang is with School of Computer and Computing Science, Hangzhou City University, Zhejiang Key Laboratory of Big Data Intelligent Computing, Hangzhou, 310015, China ({tangpp@hzcu.edu.cn}). This author's research was partly supported by the Zhejiang Provincial Natural Science Foundation of China
under Grant LMS26A010022.}, Dunbiao Niu\thanks{Dunbiao Niu  is with Department of Control
Science and Engineering, College of Electronics and Information Engineering,
National Key Laboratory of Autonomous Intelligent Unmanned Systems,
Tongji University, Shanghai 201804, China ({dunbiaoniu\_sc@163.com}).} and Aimin Xu\thanks{Aimin Xu is with Institute of Mathematics, Zhejiang Wanli University, Ningbo 315100, China ({xuaimin1009@hotmail.com}).}}

\IEEEpubid{}
\maketitle

\begin{abstract}
This paper proposes a novel distributed semismooth Newton based augmented Lagrangian method for solving a class of optimization problems over networks, where the global objective is defined as the sum of locally held cost functions, and communication is restricted to neighboring agents. 
Specifically, we employ the augmented Lagrangian method to solve an equivalently reformulated constrained version of the original problem. Each resulting subproblem is solved inexactly via a distributed semismooth Newton method. By fully leveraging the structure of the generalized Hessian, a distributed accelerated proximal gradient method is proposed to compute the Newton direction efficiently, eliminating the need to communicate with full Hessian matrices. 
Theoretical results are also obtained to guarantee the convergence of the proposed algorithm. Numerical experiments demonstrate the efficiency and superiority of our algorithm compared to state-of-the-art distributed algorithms.
\end{abstract}


\section{Introduction}
\label{sec:Introduction}
In recent years, there has been growing interest in developing distributed approaches that leverage networked nodes for data collection, storage, and processing. In response to this issue, distributed optimization models have received increasing attention in the last two decades \cite{dominguez2015distributed, yang2019survey,tsitsiklis1984problems,johansson2008sub}. 
Distributed optimization models arose with the development of multiprocessor systems and decentralized networks. These models enable collaborative problem-solving across multiple computing nodes, effectively addressing challenges such as large-scale data processing, decentralized resource allocation, and privacy preservation. They play a crucial role in enhancing system efficiency, scalability, and fault tolerance across various domains, including machine learning \cite{xia2021penalty}, wireless sensor networks \cite{wan2009event}, power systems \cite{mao2021finite}, resource allocation \cite{deng2018distributed}, and so on. 

%The distributed solution to unconstrained optimization problems involving the sum of smooth and nonsmooth functions is of great importance. Such problems are ubiquitous in various fields. For instance, in machine learning, many models involve a smooth loss function combined with a nonsmooth regularization term (e.g., L1 regularization for sparsity). In image processing, total variation regularization is a common nonsmooth term used for image denoising and reconstruction. Additionally, in operations research, certain resource allocation problems can be formulated with smooth objective functions and nonsmooth constraints. The distributed approach solving the problems can harness parallel computation across multiple nodes to enhance efficiency, making it suitable for large-scale data scenarios and demonstrating robust adaptability and fault tolerance in decentralized networks. Next, we will design an efficient distributed algorithm to address this type of problem.

In this paper, we consider the following distributed optimization problem: 

\begin{equation}\label{eq:problemBx=y1}
    \min\limits_{w\in \mathbb{R}^n}\sum_{i=1}^{m}\Big\{f_i(w)+g_i(w)\Big\},
\end{equation}
where each $f_i: \mathbb{R}^n\to  \overline{\mathbb{R}}$ is a closed proper \(\mu_i\)-strongly convex (see Definition \ref{strongly convex}) and \(L_i\)-smooth  function (see Definition \ref{L-smooth}) privately owned by the $i$-th agent, while each $g_i: \mathbb{R}^n\to  \overline{\mathbb{R}}$ is a closed proper convex but possibly nonsmooth function owned by the $i$-th agent.
In this framework, each agent in a network collaborates to minimize the aggregate of individual objective functions by performing private local computations and exchanging information with neighboring agents. However, each local agent is aware only of its own objective, which is constructed from a portion of the entire dataset.   

Given the broad applicability of problem \eqref{eq:problemBx=y1}, the development of efficient algorithms is of considerable importance. For an efficient distributed algorithm, it is essential to integrate well-designed communication protocols that facilitate efficient information exchange among local agents and their neighbors. Moreover, the local computational schemes used by these agents to process received information must be carefully structured to guarantee that all agents converge to a common optimal solution of problem \eqref{eq:problemBx=y1}.

For the case of \(g_i(w)=0\), $i=1,\cdots, m$, there are many well-known distributed algorithms, such as EXTRA \cite{shi2015extra}, DIGing \cite{nedich2017achieving}, Aug-DGM \cite{xu2015augmented}, Acc-DNGD \cite{qu2019accelerated}, DGD \cite{qu2017harnessing}, DAN \cite{zhang2022distributed}, the distributed SDD Newton method \cite{tutunov2019distributed}, DINN \cite{niu2025dual} and the distributed FGM method \cite{uribe2021dual}. Since the nonsmooth term $g_i(w)$ can naturally express requirements, such as variable selection, physical constraints, or economic penalties, the applicability of these aforementioned distributed algorithms might be relatively limited, as they require the objective function of the problem to be smooth.

%that do not consider the non-smooth term might face limitations in practical applications.

%Among them, the dual inexact nonsmooth Newton (DINN) method proposed in \cite{niu2025dual} is the most efficient. Although the DINN method is efficient, its applicability is relatively limited, as it requires the objective function of the problem to be strongly convex and smooth.
 
For the case of \(g_i(w)\neq0\) for some $i$,  numerous distributed algorithms have emerged over time, including PG-EXTRA \cite{shi2015proximal},  PG-ADMM and its stochastic variant SPG-ADMM \cite{aybat2017distributed}, the modified version of ADMM \cite{mafakheri2023distributed}, the ABC algorithm (a general unified algorithmic framework) \cite{JinmingDistributed}, the fast distributed proximal gradient (FDPG) method \cite{chen2012fast} and the m-PAPG method \cite{yaoliang}. 
A key advantage of these algorithms lies in their simplicity and ease of implementation for distributed computation. Nevertheless, the computational efficiency of first-order algorithms is often not ideal, which is an inherent limitation stemming from their gradient-based nature and convergence properties. 
%Moreover, Although decentralized algorithms have been widely studied by the control community, the lower bound has not been established until 2017 [64] and a distributed dual ascent with a matching upper bound is given in [64].

% the distributed Newton method in the framework of alternating direction method of multipliers (ADMM) \cite{9683317},

The motivation of this paper is to develop a novel distributed semismooth Newton based augmented Lagrangian (DSSNAL) method for solving problem \eqref{eq:problemBx=y1} over an undirected and connected network. Specifically, we apply the augmented Lagrangian method (ALM) to a reformulated version of problem \eqref{eq:problemBx=y1}, which introduces a local variable for each agent and enforces consensus constraints among these local variables. To solve the resulting inner subproblem, we propose a distributed inexact semismooth Newton (DiSSN) method that operates without line search. Furthermore, to circumvent the communication of full Hessian matrices, we utilize a distributed accelerated proximal gradient (DAPG) method to compute the Newton direction. By integrating these components, we establish a distributed iterative scheme with provable convergence guarantees. Moreover, we employ the DAPG method to generate an ideal initial point for the DiSSN method without line search to ensure its global convergence.

In contrast to conventional second-order algorithms, which typically require the objective function to be twice continuously differentiable, our approach relies on relatively milder assumptions. This relaxation significantly broadens its applicability of the proposed algorithm to a wider range of real-world problems. Numerical experiments further demonstrate that our algorithm substantially outperforms the FDPG method \cite{chen2012fast} and the $\mathrm{Prox}$-NIDS method (a special case of the  ABC algorithm) \cite{JinmingDistributed} in terms of computational efficiency. Additionally, the DSSNAL method exhibits broader applicability compared to the DINN method \cite{niu2025dual}.

The contributions of this paper are summarized as follows:

1) We propose a novel DSSNAL method to solve problem \eqref{eq:problemBx=y1}. To our knowledge, this is the first work that successfully integrates the semismooth Newton based augmented Lagrangian (SSNAL) framework into distributed optimization. The proposed algorithm is not only computationally efficient but also naturally aligns with the decentralized structure of the problem, enabling seamless implementation over networks.

2) To solve the inner subproblems, we innovatively employ the DAPG method in two key roles: first, to warm-start the DiSSN phase, accelerating its convergence; and second, to compute the Newton direction without requiring full Hessian communication. This dual application of the DAPG method ensures both numerical efficiency and communication efficiency, facilitating practical and scalable distributed execution.

The rest of the paper is organized as follows. Section \ref{sec:preliminaries} is devoted to preliminary knowledge. In Section \ref{sec:problem statement}, problem \eqref{eq:problemBx=y1} is formally stated with necessary assumptions and reformulated via consensus constraints. The proposed DSSNAL method for this networked problem is detailed in Section \ref{sec:DSSNAL}. Section \ref{sec:Numerical experiments} presents numerical experiments to demonstrate the algorithm’s efficiency. Finally, conclusions are provided in Section \ref{sec:conclusion}.

The notations used in this paper follow conventional mathematical conventions. 
Let $\mathbb{R}^{n\times m}$ denote the space of $n\times m$ real matrices and $\mathbb{R}^n$ the $n$-dimensional real Euclidean space. We use $I_n$ to represent the $n\times n$ identity matrix and $\mathbf{1}_n \in \mathbb{R}^n$ to represent the all-ones column vector. The Euclidean $\ell_{2}$ norm of a vector $x\in\mathbb{R}^{n}$ is denoted by $\|x\|$, while for a matrix $A \in \mathbb{R}^{n\times m}$, $\|A\|$ represents its Frobenius norm.
For any matrix $A \in \mathbb{R}^{n\times m}$, 
$\operatorname{Ran}(A) := \{Ax \mid x \in \mathbb{R}^m\}$ denotes its range space, 
$\operatorname{Null}(A) := \{x \in \mathbb{R}^m \mid Ax = 0\}$ represents its null space. 
Given a set of vectors $\{v_i\}_{i=1}^s$, $\operatorname{Span}\{v_1,\ldots,v_s\}$ indicates their linear span. The Kronecker product between matrices is denoted by $\otimes$. For any real number $a$, $\lceil a \rceil$ signifies the ceiling function (the smallest integer greater than or equal to $a$). The notation $A \succeq B$ means $A-B$ is positive semidefinite.

\section{preliminaries}
\label{sec:preliminaries}

In this section, we provide some preliminary knowledge that will be utilized throughout the paper. 

Let $\mathcal{X}$ and $\mathcal{Y}$ be two real finite-dimensional Euclidean spaces, $p:\mathcal{X}\to\overline{\mathbb{R}}$ be a closed proper convex function. The conjugate function of $p$
is defined by 
\begin{eqnarray*}
p^{*}(v):=\sup_{x \in\operatorname{dom}(p)}\Big\{\langle x,v\rangle - p(x)\Big\}.
\end{eqnarray*}
The proximal mapping $\mathrm{Prox}_{p}(x)$ is defined by
\begin{equation*}
     \operatorname{Prox}_{p}(x):=\mathop{\rm argmin}\limits_{y\in\mathcal{X}}\Big\{p(y)+\frac{1}{2}\|y-x\|^2\Big\}\ ,\ x\in\mathcal{X},
\end{equation*}
which satisfies the Moreau identity (see e.g. \cite[Theorem 31.5]{Rockafellar70}) 
\begin{equation*}
 \mathrm{Prox}_{tp}(x)+t\mathrm{Prox}_{p^{*}/t}(x/t)=x, 
\end{equation*}
where $x\in\mathcal{X}$ and the parameter $t>0$. 

A function $F:\mathcal{X}\rightarrow\mathcal{Y}$ is called directionally differentiable at a point $x\in\mathcal{X}$ in a direction $v\in\mathcal{X}$ if the limit
\begin{eqnarray*}
	F'(x,v):=\lim_{t\downarrow 0}\frac{F(x+tv)-F(x)}{t}
\end{eqnarray*}
exists. If $F$ is directionally differentiable at $x$ in every direction $v\in\mathcal{X}$, we say that $F$ is directionally differentiable at $x$.

Let $\mathcal{U}\subseteq\mathcal{X}$ be an open set and $F:\mathcal{U}\rightarrow \mathcal{Y}$ be a locally Lipschitz continuous function with $\mathcal{U}_{F}$ the subset of $\mathcal{U}$ where $F$ is F(r\'{e}chet)-differentiable. Denote $F'(x)$  the Jacobian of $F$ at $x\in\mathcal{U}_{F}$. It is known from \cite[Theorem 9.60]{Rockafellar98} that the set $\mathcal{U}\setminus\mathcal{U}_{F}$ is negligible. For any $x\in\mathcal{U}$, define the B-subdifferential of $F$ at $x$ by
\begin{align*}
	\partial_{B}F(x):=\Big\{V\ \Big|\ \exists\ x^{k}\rightarrow x\ \mbox{with}\ x^{k}\in\mathcal{U}_{F}&\\ \quad\mbox{and}\ F'(x^{k})\rightarrow V\Big\}.&
\end{align*}
The Clarke subdifferential of $F$ at $x$ is given by $\partial F(x):=\mbox{conv}(\partial_{B}F(x))$, which is the convex hull of $\partial_{B}F(x)$.

Next, we recall three important concepts about the strong convexity and smoothness of functions. 

\begin{definition}
\label{strongly convex}
(\cite{Nesterov2018}) A differentiable function $f:\mathcal{X}\rightarrow \mathbb{R}$ is called $\mu$-\textit{strongly convex} on $\mathcal{X}$ if there exists a constant $\mu>0$ such that for any $x,y\in \mathcal{X}$ we have
\begin{align*}
f(y) \geq f(x)+\langle \nabla f(x),y-x\rangle + \frac{1}{2}\mu\|y-x\|^2. 
\end{align*}
The constant $\mu$ is called the \textit{convexity parameter} of function $f$.

\end{definition}

\begin{definition}\label{semismooth}
 (\cite{QiSun1993}) Let $F: O \subseteq \mathcal{X} \rightarrow \mathcal{Y}$ be a locally Lipschitz continuous function on an open set $O$. $F$ is said to be \textit{(strongly) semismooth} at $x \in O$, if $F$ is directionally differentiable at $x$ and for any $V \in \partial F(x + \Delta x)$ with $\Delta x \rightarrow 0$, 
\[
F(x + \Delta x) - F(x) - V \Delta x = o(\|\Delta x\|) \ (O(\|\Delta x\|^{2})). 
\]
We say that $F$ is a (strongly) semismooth function on $O$, if it is (strongly) semismooth everywhere on $O$.   
\end{definition}

\begin{definition}(\cite{Nesterov2018})\label{L-smooth}
Let $L \geq 0$. A function $f: \mathcal{X} \rightarrow \overline{\mathbb{R}}$ is said to be \textit{$L$-smooth} over a set $D \subseteq \operatorname{dom}f$ if it is differentiable over $D$ and satisfies
\begin{equation}
\| \nabla f(x) - \nabla f(y) \| \leq L \| x - y \| \quad \text{for all } x, y \in D.
\end{equation}

\end{definition}

Let $\mathcal{F}: \mathcal{X} \rightrightarrows \mathcal{Y}$ be a multivalued mapping with its inverse as $\mathcal{F}^{-1}$. Define the graph of $\mathcal{F}$ as
\begin{eqnarray*}
\mathrm{gph}\mathcal{F}:=\{(x,y)\in \mathcal{X}\times\mathcal{Y}\,|\,y\in \mathcal{F}(x)\}.    
\end{eqnarray*}

Now we present  two definitions which are very important for analyzing the convergence rate of our algorithm in Section
\ref{sec:DSSNAL}.

\begin{definition}
(\cite{luque1984asymptotic}) Let $\mathcal{F}: \mathcal{X} \rightrightarrows \mathcal{Y}$ be a multivalued mapping and $y \in \mathcal{Y}$ satisfy $\mathcal{F}^{-1}(y) \neq \emptyset$. $\mathcal{F}$ is said to satisfy the \textit{error bound} condition at $y$ with modulus $\kappa \geq 0$ if there exists $\varepsilon > 0$ such that for every $x \in \mathcal{X}$ with $\operatorname{dist}(y, \mathcal{F}(x)) \leq \varepsilon$,
\begin{equation}
\label{eq:error bound}
\operatorname{dist} \left( x, \mathcal{F}^{-1}(y) \right) \leq \kappa \, \operatorname{dist}(y, \mathcal{F}(x)).
\end{equation}   
\end{definition}

\begin{definition}(\cite{dontchev2009implicit})\label{metric subregularity}
Let $\mathcal{F}: \mathcal{X} \rightrightarrows \mathcal{Y}$ be a multivalued mapping and $(\bar{x}, \bar{y}) \in \operatorname{gph}\mathcal{F}$. $\mathcal{F}$ is said to be \textit{metrically subregular} at $\bar{x}$ for $\bar{y}$ with modulus $\kappa \geq 0$ if there exist neighborhoods $U$ of $\bar{x}$ and $V$ of $\bar{y}$ such that
\begin{equation*}
\label{eq:metric subregularity}
\operatorname{dist}(x, \mathcal{F}^{-1}(\bar{y})) \leq \kappa \, \operatorname{dist}(\bar{y}, \mathcal{F}(x) \cap V) \quad \forall\, x \in U.
\end{equation*} 
\end{definition}

\section{Problem statement}
\label{sec:problem statement}
In this section, we provide a detailed description of the problem under consideration.

\subsection{Problem reformulation}
\label{subsec:problem formulation}
Let $x:=[x^T_1,x^T_2,\dots,x^T_m]^T\in  \mathbb{R}^{mn}$ and      $\tilde{y}:=[\tilde{y}^T_1,\tilde{y}^T_2,\dots,\tilde{y}^T_m]^T\in  \mathbb{R}^{mn}$ be stacked column vectors, where $x_i\in  \mathbb{R}^n $ and  $\tilde{y}_i\in  \mathbb{R}^n$\(\) are local variations at the $i$-th agent. Problem \eqref{eq:problemBx=y1} is written equivalently as 
\begin{equation}
\begin{aligned}
    &\min_{\substack{x \in \operatorname{Ran}(\mathbf{1}_m \otimes \mathbf{I}_n) \\
    \tilde{y} \in \operatorname{Ran}(\mathbf{1}_m \otimes \mathbf{I}_n)
    }} \quad &&\sum_{i=1}^m \Big\{f_i(x_i) + g(\tilde{y}_i)\Big\} \\
    &\quad\quad\operatorname{s.t.} && \quad x = \tilde{y}.
\end{aligned}
\label{eq:promblemWx=0}
\end{equation}

We make the following assumption which is used in the rest of the paper.
\begin{assumption}\label{assump1}
 For every $i\in\{1,\ldots,m\}$, both  \(\nabla f_i\) and  $\mathrm{Prox}_{g_i}$ are strongly semismooth.    
\end{assumption}

\subsection{Model description}
\label{subsec:Model description}
Let \(\mathcal{G}=(\mathcal{V},\mathcal{E})\) be a communication graph of agents, where \(\mathcal{V}=\{1,\dots,m\}\) is the set of identity numbers for $m$ agents, \(\mathcal{E}\subseteq\mathcal{V}\times\mathcal{V}\) is the set of edges,  \((i,j)\in\mathcal{E}\) indicates that the $i$-th agent can communicate with the $j$-th agent. Denote $\mathcal{N}_i=\{j\in\mathcal{V}\,|\,(i,j)\in\mathcal{E}\}$ as the set of neighbors for the $i$-th agent. The graph $\mathcal{G}$ is \textit{undirected} if every edge $(i,j)\in\mathcal{E}$ is undirected. The graph $\mathcal{G}$ is \textit{connected} if for each $i,j\in\mathcal{V}$ $(i\neq j)$ there exists a sequence $\{p_{1},\ldots,p_{k}\}$ with $p_{1}=i$, $p_{k}=j$ and $(p_{t},p_{t+1})\in\mathcal{E}$ for $t=1,\ldots,k-1$. In the following section, we adopt an assumption, which is widely used in distributed optimization (see e.g., \cite{niu2025dual} and \cite{JinmingDistributed}). 

\begin{assumption}\label{assump3}
    The graph \(\mathcal{G}\) is undirected and connected.
\end{assumption}

Then we present the definition of a matrix associated with the graph $\mathcal{G}$, which is useful in designing distributed optimization algorithms.
\begin{definition}\label{defL}
  A matrix \(L \in  \mathbb{R}^{m\times m}\) is called the gossip matrix of a graph \(\mathcal{G}\) if it satisfies the following properties: \label{defL}
\begin{enumerate}
    \item[(a)] \textbf{Symmetry}: \( L^T = L \);     
    \item[(b)] \textbf{Positive semi-definite}: \( L \succeq 0 \);      
    \item[(c)] \textbf{Connectivity}: \( \operatorname{Null}(L) = \operatorname{Span}\{\mathbf{1}_m\} \);         
    \item[(d)] \textbf{Graph induced}: \( L_{ij} \neq 0 \) only if \( i = j \) or \( (i, j) \in \mathcal{E} \).      
   \end{enumerate}  
\end{definition}

Based on Definition \ref{defL}, we can define a gossip matrix as

\begin{align*} 
    W:=L\otimes I_n\in  \mathbb{R}^{nm\times nm}, 
\end{align*}
where $L$ is the gossip matrix given by Definition \ref{defL}. Then we can easily obtain that \(Wx=0\) if and only if \(x_1=\dots=x_m\), i.e., \(\operatorname{Null}(W)=\operatorname{Ran}(\mathbf{1}_m\otimes I_n)\). Thus problem \eqref{eq:promblemWx=0} can be equivalently rewritten as 

\begin{equation}
\begin{split}
&\min_{{x_i\in\mathbb{R}^{n}}\atop{i=1,\ldots,m}} \quad \sum_{i=1}^{m}\Big\{f_i(x_i)+g(y_i)\Big\} \\
&\,\,\operatorname{s.t.}\quad Wx=0,\\
&\phantom{\operatorname{s.t.}\quad}\,\, x_i=y_i,\quad i=1,\dots,m.
\end{split}
\label{eq:promblemx=y}
\end{equation}

For simplicity, we denote $F(x):=\sum\limits_{i=1}^{m}f_i(x_i)$, $G(y):=\sum\limits_{i=1}^{m}g_i(y_i)$ and \(B=\begin{pmatrix}I_{mn}\\W\end{pmatrix}\), where $x=[x_1^T,\dots,x_m^T]^T\in  \mathbb{R}^{mn}$ and $y=[y_1^T,\dots,y_m^T,0^T,\dots0^T]^T\in \mathbb{R}^{2mn}$. Therefore, \eqref{eq:promblemx=y} becomes

\begin{equation}
\begin{split}
&\min_{x \in \mathbb{R}^{mn}}  \quad F(x)+G(y)\\
&\,\,\operatorname{s.t.}\quad\quad Bx=y,
\end{split}
\label{eq:problemBx=y}
\end{equation}
where \(F(x)\) is \(\mu\)-strongly convex and $\nu$-smooth, \(\mu=\min\limits_{1\leq i\leq m}\,\mu_i\) and \(\nu=\max\limits_{1\leq i\leq m}\,L_i\). 

The Lagrangian function of problem \eqref{eq:problemBx=y} is
\begin{equation*}
    l(x,y;\lambda)=F(x)+G(y)- \langle \lambda, Bx-y \rangle.
\end{equation*}
Then we can obtain its Karush-Kuhn-Tucker (KKT) condition 
\begin{align*}
    \begin{array}{cc} 
     \nabla F(x)-B^T\lambda=0,  \\
     Bx-y=0,         \\
     0\in\partial G(y)+\lambda.
\end{array}
\end{align*}

The dual of problem \eqref{eq:problemBx=y} is
\begin{equation}\label{H()}
  \min_{\lambda \in \mathbb{R}^{2mn}}  \quad H(\lambda):=F^*(B^T\lambda)+G^*(-\lambda),  
\end{equation}
where
\begin{align*}
   F^*(z)= \sum_{i=1}^m f^*_i(z_i),  \quad z = \begin{pmatrix}
         z_1 \\
        \vdots \\
        z_{m}
    \end{pmatrix} \in \mathbb{R}^{mn}, 
\end{align*}
\begin{align*}
    G^*(\lambda)= \sum_{i=1}^m g^*_i(\lambda_i),  \quad \lambda = \begin{pmatrix}
         \lambda_1 \\
        \vdots \\
        \lambda_{2m}
    \end{pmatrix} \in \mathbb{R}^{2mn}.
\end{align*}

We can define the operators \( \mathcal{T}_H \) and \( \mathcal{T}_l \) related to the closed proper convex function $H$ and the convex-concave function $l$, respectively by
\begin{align*}
     \mathcal{T}_H(\lambda) :&= \partial H(\lambda),\\
     \mathcal{T}_l(x, y, \lambda) :&= \{(x', y', \lambda') \mid (x', y', -\lambda') \in \partial l(x, y; \lambda)\}.
\end{align*}

\section{DSSNAL method}
\label{sec:DSSNAL}
In this section, we apply the DSSNAL method to solve problem \eqref{eq:problemBx=y}.
\subsection{ALM for problem \eqref{eq:problemBx=y}}
\label{subsec:DSSNAL}
To begin with, we provide the framework of ALM on problem \eqref{eq:problemBx=y}. For fixed \(\sigma>0\) and \(\lambda\in  \mathbb{R}^{2mn}\),  the augmented Lagrangian function associated with problem \eqref{eq:problemBx=y} is
\begin{align*}
    {L}_{\sigma}(x, y; \lambda) \nonumber 
= F(x)+G(y)+\frac{\sigma}{2} \left\| Bx-y-\frac{1}{\sigma}\lambda \right\|^2-\frac{1}{2\sigma}\|\lambda\|^2,
\end{align*}
where  $x\in  \mathbb R^{mn}$ and $y\in \mathbb R^{2mn}$. Then, given a positive sequence $\{\sigma_{k}\}$ with $\sigma_{k}\to +\infty$ as $k\to +\infty$ and choose an initial point $(x^0, y^0, \lambda^0) \in \operatorname{int}(\operatorname{dom} F) \times \operatorname{dom} G \times \mathbb{R}^{2mn}$, ALM for problem \eqref{eq:problemBx=y} is outlined as follows: For $k=0,1,\cdots$,

\begin{subequations}\label{AL_method}
		\begin{align}
&(x^{k+1}, y^{k+1}) \approx \mathop{\rm argmin}_{x,y} \{ \psi_k(x,y) := L_{\sigma_k}(x, y; \lambda^k) \}, \label{eq:subproblem}\\
&\lambda^{k+1} = \lambda^k - \sigma_k(Bx^{k+1} - y^{k+1}).\label{dual_update}
		\end{align}
	\end{subequations}

% \begin{algorithm}[H]\small
% \caption{AL method} \label{alg:SSNAL}
% Given $\sigma_{0} > 0$, choose $(x^0, y^0, \lambda^0) \in \operatorname{int}(\operatorname{dom} F) \times \operatorname{dom} G \times \mathbb{R}^{2mn}$. For $k=0,1,\cdots$, iterate:
% \begin{description}
%     \item [Step 1.] Compute 
%     \begin{equation}
%         (x^{k+1}, y^{k+1}) \approx \mathop{\rm argmin}_{x,y} \{ \psi_k(x,y) := L_{\sigma_k}(x, y; \lambda^k) \}. \label{eq:subproblem}
%     \end{equation}
%     \item [Step 2.] Compute 
%     \begin{equation*}
%         \lambda^{k+1} = \lambda^k - \sigma_k(Bx^{k+1} - y^{k+1})
%     \end{equation*}
%     and update $\sigma_{k+1}=2\sigma_{k}$. If the desired stopping criterion is satisfied, terminate; otherwise, set $k:=k+1$ and go to Step 1.
% \end{description}
% \end{algorithm}

Since the inner subproblem \eqref{eq:subproblem} can't be exactly solved. 
To ensure the computational efficiency of ALM, we employ the following stopping criteria proposed by Rockafellar (1976) for its approximate solution:
\begin{align*}
  \text{(A)} & \ \psi_k(x^{k+1},y^{k+1})-\mathrm{inf}\,\psi_k\leq\frac{\epsilon^2_k}{2\sigma_k},\,\quad\sum_{k=0}^{\infty}\epsilon_k<+\infty,     \\
   \text{(B)} & \ \psi_k(x^{k+1}, y^{k+1}) - \inf \psi_k \leq \frac{\delta_k^2}{2\sigma_k} \|\lambda^{k+1} - \lambda^k\|^2, \\ &\ \sum_{k=0}^{\infty} \delta_k < +\infty,\\
   \text{(C)} & \quad \text{dist}(0, \partial \psi_k(x^{k+1}, y^{k+1})) \leq \frac{\delta'_k}{\sigma_k} \|\lambda^{k+1} - \lambda^k\|, \\
   &\ 0 \leq \delta'_k \rightarrow 0.
\end{align*}

For subproblem \eqref{eq:subproblem}, we omit the superscripts and subscripts when discussing the algorithm itself, and 
define
\begin{align*}
   \phi(x) 
     &:=\min_{y}\psi(x,y)\\
     &=F(x)+G(\mathrm{Prox}_{\frac{1}{\sigma}G}(Bx-\frac{1}{\sigma}\lambda))+\\
     &\ \frac{1}{2\sigma}\|\mathrm{Prox}_{\sigma G^*}(\sigma Bx-\lambda)\|^2-\frac{1}{2\sigma}\|\lambda\|^2, 
\end{align*}
then
\begin{equation}
    \nabla\phi(x)=\nabla F(x)+B^T\mathrm{Prox}_{\sigma G^*}(\sigma Bx-\lambda).\label{eq:nalbaphi}
\end{equation}  

\begin{theorem}
\label{Fproperty}
$\phi(x)$ is $\mu$-strongly convex and $L$-smooth, where $\mu=\min\limits_{1\leq i\leq m}\mu_i$, and $L=\max\limits_{1\leq i\leq m}L_i+\sigma\|B\|^2$.    
\end{theorem}
\begin{proof}
Since \(f_i\) is $\mu_i$-strongly convex and $F(x)=\sum\limits_{i=1}^mf_i(x_i)$, we can get that $F(x)$ is $\mu$-strongly convex by  Lemma 3.1 in \cite{niu2025dual}. Since $G$ is a proper convex function,  $G^*$ is a closed proper convex function by \cite[Theorem 12.2]{Rockafellar70}. Then by \cite[Theorem 6.42]{Beck2017} the proximal operator $\mathrm{Prox}_{\sigma G^*}$ is firmly non-expansive, i.e., for any $u,v\in \mathbb{R}^{2mn}$,
\begin{align*}
\|\mathrm{Prox}_{\sigma G^*}(u)-\mathrm{Prox}_{\sigma G^*}(v)\|^2 \leq \langle \mathrm{Prox}_{\sigma G^*}(u)-\mathrm{Prox}_{\sigma G^*}(v), u-v \rangle.    
\end{align*}
Thus, in combination of the strong convexity of $F$, we obtain
\begin{align*}
   (\nabla\phi(x) - \nabla\phi(y))^T(x-y)&=(\nabla F(x) - \nabla F(y))^T(x-y)\\&\quad +(\mathrm{Prox}_{\sigma G^*}(\sigma Bx-\lambda) -\\
   &\quad\mathrm{Prox}_{\sigma G^*}(\sigma By-\lambda))^T(Bx-By)\\
   &\geq \mu\|x-y\|^2.
\end{align*}
Thus, $\phi$ is $\mu$-strongly convex.   

Due to \eqref{eq:nalbaphi} and the non-expansiveness of the proximal operator, we can deduce
\begin{align*}
    \|\nabla\phi(x) - \nabla\phi(y)\| &\leq \|\nabla F(x) - \nabla F(y)\| + \\
    &\quad\sigma \|B^T\| \|B\| \|x - y\| \\
    &\leq (\max_{1 \leq i \leq m} L_i + \sigma \|B\|^2) \|x - y\|.
\end{align*}

Thus, we have that \(\phi(x)\) is $L$-smooth.   
\end{proof}

Since the function \(\phi_{k}\) is \(\mu\)-strongly convex, the following inequality holds 
\begin{align*}
    \psi_k(x^{k+1},y^{k+1})-\inf\,\psi_k&= \phi_k(x^{k+1})-\inf\,\phi_k\\
            &\leq\frac{\|\nabla\phi_k(x^{k+1})\|^2 }{2\mu},
\end{align*}
where $y^{k+1}=\operatorname{Prox}_{\frac{1}{\sigma_k}G}(Bx^{k+1}-\frac{1}{\sigma_k}\lambda^k)$, $ (\nabla \phi_k(x^{k+1}), 0) \in \partial \psi_k(x^{k+1}, y^{k+1})$. Thus, we replace the aforementioned stopping criteria (A), (B) and (C) with the following implementable stopping criteria:
\begin{align*}
  \text{(A$'$)} & \ \|\nabla\phi_k(x^{k+1})\|^2\leq\frac{\epsilon^2_k\mu}{\sigma_k},\,\quad\sum_{k=0}^{\infty}\epsilon_k<+\infty,     \\
   \text{(B$'$)} & \ \|\nabla\phi_k(x^{k+1})\|^2 \leq \frac{\delta_k^2\mu}{\sigma_k} \|\lambda^{k+1} - \lambda^k\|^2, \ \sum_{k=0}^{\infty} \delta_k < +\infty,\\
   \text{(C$'$)} & \ \|\nabla\phi_k(x^{k+1})\| \leq\frac{\delta_k'}{\sigma_k}  \|\lambda^{k+1} - \lambda^k\|, \ 0 \leq \delta'_k \rightarrow 0.
\end{align*}

Now, we state the global convergence of ALM %Algorithm \ref{alg:SSNAL}
for problem \eqref{eq:problemBx=y}, which is based on \cite{Rockafellar1976}.
\begin{theorem}\label{convergenceDSSNAL}
Suppose that the solution set of problem \eqref{eq:problemBx=y} is nonempty. Let $\{(x^k,y^k,\lambda^k)\}$ be the infinite sequence generated by ALM %Algorithm \ref{alg:SSNAL}
with stopping criteria (A$'$) and. Let $(x^*, y^*)$ be the unique optimal solution of problem \eqref{eq:problemBx=y}. Then the sequence $\{(x^k,y^k)\}$ converges to the unique optimal solution $(x^*, y^*)$ and $\{\lambda^k\}$ converges to an optimal dual solution $\lambda^*$ of problem \eqref{eq:problemBx=y}. 

\end{theorem}
\begin{proof}
Since the objective function in problem \eqref{eq:problemBx=y1} is strongly convex, the optimal value of problem \eqref{eq:problemBx=y1} is finite. Due to the equivalence between problem \eqref{eq:problemBx=y1} and problem \eqref{eq:problemBx=y}, the optimal value of problem \eqref{eq:problemBx=y} is finite. Since at least the zero vector is a feasible solution of problem \eqref{eq:problemBx=y}, by \cite[Corollary 31.2.1]{Rockafellar70}, we have that the solution set of the dual problem \eqref{H()} is nonempty, and the optimal value of \eqref{H()} is finite and equal to the optimal value of its primal problem \eqref{eq:problemBx=y}. That is, the solution set to the KKT system associated with \eqref{eq:problemBx=y} and \eqref{H()} is nonempty. 
According to \cite[Theorem 4]{Rockafellar1976}, we can obtain the boundedness of $\{\lambda^k\}$ and $\{(x^k,y^k)\}$, and further their convergence results.
\end{proof}

 \begin{theorem}\label{convergence2}
Suppose that the solution set of problem \eqref{eq:problemBx=y} is nonempty. Suppose that $\mathcal{T}_H$ satisfies the error bound condition \eqref{eq:error bound} for the origin with modulus $a_H$. Let $\{(x^k, y^k, \lambda^k)\}$ be any infinite sequence generated by ALM %Algorithm \ref{alg:SSNAL}
with stopping criteria (A$'$) and (B$'$). Then, the sequence $\{\lambda^k\}$ converges to $\lambda^* \in \Omega$ and for all $k$ sufficiently large,
\begin{equation}
\operatorname{dist}(\lambda^{k+1}, \Omega) \leq \theta_k \operatorname{dist}(\lambda^k, \Omega),
\end{equation}
where $\theta_k = (a_H(a_H^2 + \sigma_k^2)^{-1/2} + 2\delta_k)(1 - \delta_k)^{-1} \rightarrow \theta_\infty = a_H(a_H^2 + \sigma_\infty^2)^{-1/2} < 1$, $\sigma_k\rightarrow\sigma_{\infty}$ as $k \rightarrow +\infty$. Moreover, the sequence $\{(x^k, y^k)\}$ converges to the unique optimal solution $(\hat{x}, \hat{y}) \in \operatorname{int}(\operatorname{dom } F) \times \operatorname{dom } G$ to \eqref{eq:problemBx=y}. 

Moreover, if $\mathcal{T}_l$ is metrically subregular at $(x^*, y^*, \lambda^*)$ for the origin with modulus $a_l$ and the stopping criterion (C$'$) is also used, then for all $k$ sufficiently large,
\begin{equation}
\label{eq:xy-convergence rate}
\|(x^{k+1},y^{k+1}) - (x^*, y^*)\| \leq \theta'_k \|\lambda^{k+1} - \lambda^k\|,
\end{equation}
where $\theta'_k = a_l(1 + \delta'_k)/\sigma_k$ with $\lim_{k \rightarrow \infty} \theta'_k = a_l/\sigma_\infty$.
\end{theorem}
\begin{proof}
From \cite[Theorem 2.1]{luque1984asymptotic}, \cite[Proposition 7, Theorem 5]{Rockafellar1976} and Theorem \ref{convergenceDSSNAL}, we can obtain the results about the convergence and convergence rate of $\{\lambda^k\}$. Since $\mathcal{T}_l$ is metrically subregular at $(x^*, y^*, \lambda^*)$ for the origin with modulus $a_l$ and $(x^k, y^k, \lambda^k)\rightarrow(x^*, y^*, \lambda^*)$, then for all $k$ sufficiently large, 
\begin{eqnarray*}
&\|(x^{k+1},y^{k+1})-(x^*,y^*)\| + \mathrm{dist}(\lambda^{k+1},\Omega)\\
\leq & a_l\mathrm{dist}(0,\mathcal{T}_l(x^{k+1},y^{k+1},\lambda^{k+1})).
\end{eqnarray*}
Hence, by (4.21) in \cite{Rockafellar1976} and the stopping criterion (C), we have that for all $k$ sufficiently large,
\begin{equation*}
\|(x^{k+1},y^{k+1}) - (x^*, y^*)\| \leq a_l(1 + \delta'_k)/\sigma_k \|\lambda^{k+1} - \lambda^k\|.
\end{equation*}
\end{proof}
\begin{remark}
In fact, from \eqref{eq:xy-convergence rate} we can easily prove that for all $k$ sufficiently large,
\begin{align}
\|(x^{k+1}, y^{k+1}) - (x^*, y^*)\| 
&\leq \theta'_k (1 - \delta_k)^{-1} \operatorname{dist}(\lambda^k, \Omega) \label{(xk+1,yk+1)}
\end{align}
where $\theta'_k (1 - \delta_k)^{-1} \rightarrow a_l/\sigma_\infty$. Therefore, if $\sigma_\infty = \infty$, the Q-superlinear convergence of $\{x^k\}$ and \eqref{(xk+1,yk+1)} further imply the R-superlinear convergence of $\{(x^k, y^k)\}$.
\end{remark}

\subsection{DiSSN method for problem \eqref{eq:subproblem}}
\label{subsec:SSN method}
In this part, we focus on how to solve the inner subproblem \eqref{eq:subproblem}. Note that problem \eqref{eq:subproblem} is equivalent to 
\begin{equation}
    \min_{x} \phi(x).\label{eq:minphi} 
\end{equation}
It follows from Theorem \ref{Fproperty} that $\phi$ is $\mu$-strongly convex and $L$-smooth. Therefore, problem \eqref{eq:minphi} has a unique solution which satisfies the following nonsmooth equation:
\begin{equation*}
    \nabla \phi(x)=0.\label{nonsmoothequation} 
\end{equation*}

First, we define the following multifunction
\begin{equation}
    \hat{\partial}^2\phi(x):=\partial(\nabla F)(x)+\sigma B^T\partial\mathrm{Prox}_{\sigma G^*}(\sigma Bx-\lambda)B, \label{eq:newtonphi}
\end{equation}
which, from \cite{hiriat1984generalized}, has the following property
\begin{equation*}
     \partial^2 \phi(x)(d) = \hat{\partial}^2 \phi(x)(d) \quad \forall \, d \in \operatorname{int}(\operatorname{dom}F).
\end{equation*}
Therefore, all matrices in the set $\hat{\partial}^2\phi(x)$ can be viewed as the generalized Hessian matrix of $\phi$ at the point $x$, which motivates us to apply the DiSSN method to solve problem \eqref{eq:minphi}.

Next, given a sequence $\{\eta_{t}\}$ with $\eta_{t} \to 0$ as $t \to +\infty$. Initialize ${x}^{0} \in \mathbb{R}^{mn}$, the framework of the DiSSN method is outlined as follows: For $t=0,1,\cdots$,
$${x}^{t+1} = {x}^{t} + {d}^t,$$
where $d^t\in\mathbb{R}^{mn}$ is the approximate Newton direction  satisfying the condition
\begin{equation}
    \label{eq:stopping criterion}
    \| M d^t+ \nabla \phi(x^t) \| \leq {\eta_t}\|\nabla \phi(x^t)\|,
    \end{equation}
here
\begin{equation*}
     M:=V+\sigma B^THB\in\hat{\partial}^2 \phi(x^t),\label{eq:M}
\end{equation*} 
with \(V\in\partial(\nabla F)(x^t)\) and \(H\in\partial\mathrm{Prox}_{\sigma G^*}(\sigma Bx^t-\lambda)\).

To get an approximate Newton direction satisfying the condition \eqref{eq:stopping criterion}, the APG method is employed, which can be summarized as the following iterative process: For $j=0,1,\dots$,
\begin{subequations}\label{APG_method}
		\begin{align}
&\tilde{d}^{j} = d^{j} + \beta(d^{j} - d^{j-1}),\label{uPDATA_1}\\
&d^{j+1} = \tilde{{d}}^{j} - \frac{1}{L}(M \tilde{{d}}^{j}+\nabla\phi(x^t)),\label{eq:APG-gradient}
		\end{align}
	\end{subequations}
where $\beta = \frac{\sqrt{L} - \sqrt{\mu}}{\sqrt{L} + \sqrt{\mu}}$ and $d^{0} = d^{-1} = \mathbf{0} \in \mathbb{R}^{mn}$. According to Lemma 4.3 in \cite{niu2025dual}, this algorithm can produce the required Newton direction after the $N(\eta_t)$-th iteration, where
\begin{equation*}
     N(\eta_t):=\Big\lceil\frac{2\mathrm{ln}(\frac{1}{\eta_t}\sqrt{\frac{2L}{\mu}})}{\mathrm{ln}(\frac{1}{1-\sqrt{\frac{\mu}{L}}})}\Big\rceil.\label{eq:firstN}
\end{equation*}
Since the matrix $M\in\hat{\partial}^2\phi(x^t)$ has some special structure, we need to investigate the details about the step update \eqref{eq:APG-gradient} as below.

\begin{theorem}
Assuming Assumptions \ref{assump1} and \ref{assump3} hold, the gradient step \eqref{eq:APG-gradient} in the APG method can be reformulated in a distributed form. For the $i$-th agent $(i=1\dots m)$, the  computation is given as follows:
\begin{align*}
    d_i^{j+1} &= \tilde{{d}}_{i}^{j} - \frac{1}{L} \Biggl[ V_i\tilde{d}_i^{j} + \sigma H_i \tilde{d}_i^{j}+ \sigma \sum_{k \in \mathcal{N}_i} L_{ik} \hat{d}_k^{j}  \\
    & +\nabla f_i(x_i^t) + T_i +\sum_{k \in \mathcal{N}_i} L_{ik} u_{k+m} \Biggr],
    \end{align*} 
where ${V}_{i} \in \partial \nabla f_i(x_i^t)$, $T_i = \mathrm{Prox}_{\sigma g_i^*}(u_i)$, $H_i \in \partial \mathrm{Prox}_{\sigma g_i^*}(u_i)$, $\tilde{{d}}_{i}^{j} = {d}_{i}^{j} + \beta \left( {d}_{i}^{j} - {d}_{i}^{j-1} \right)$, $ u_i = \sigma x_i^t - \lambda_i$, $ u_{i+m} = \sigma \sum_{k \in \mathcal{N}_i} L_{ik} x_k^t - \lambda_{i+m}$ and $\hat{d}_i^{j} = \sum_{k \in \mathcal{N}_i} L_{ik} \tilde{d}_k^{j} $.
\end{theorem}
\begin{proof}
 From \eqref{eq:problemBx=y}, we note that
\[
\partial(\nabla F)(x^t) = \left\{ \begin{array}{l}
V = \begin{pmatrix}
V_1 & \cdots & 0 \\
\vdots & \ddots & \vdots \\
0 & \cdots & V_m
\end{pmatrix} \in \mathbb{R}^{nm \times nm} \, \Bigg| \\
V_i \in \partial(\nabla f_i)(x_i^t),\ \forall i = 1, \cdots, m
\end{array} \right\}.
\]
Meanwhile,
\begin{align}
\mathrm{Prox}_{\sigma G^*}(u) &=\mathop{\rm argmin}_{\lambda\in\mathbb{R}^{2mn}}\Big\{\sum_{i=1}^mg^*_i(\lambda_i)+\frac{1}{2\sigma}\sum_{i=1}^{2m}\|\lambda_i-u_i\|^2\Big\}\nonumber\\
    &=\begin{pmatrix}
        \mathrm{Prox}_{\sigma g^*_1}(u_1)\\
        \vdots\\
        \mathrm{Prox}_{\sigma g^*_m}(u_m)\\
        u_{m+1}\\
        \vdots\\
         u_{2m}
    \end{pmatrix},
    \label{eq:Prox}
\end{align}
where $u = \sigma Bx^t - \lambda$.
In addition,
\begin{align*}
   \partial\mathrm{Prox}_{\sigma G^*}(u)= \begin{pmatrix}
 \partial\mathrm{Prox}_{\sigma g_1^*}(u_1) &  &  &\\
 & \ddots &  &\\
 &  & \partial\mathrm{Prox}_{\sigma g_m^*}(u_m)&\\
 &&& I_{mn}
\end{pmatrix}.
\end{align*}
From \eqref{eq:nalbaphi}, \eqref{eq:Prox} and the structure of $B$, we have
\begin{align*}
   \nabla\phi(x^t) =\begin{pmatrix}
       \nabla f_1(x_1^t)\\ \vdots\\ \nabla f_m(x_m^t)
   \end{pmatrix}+\begin{pmatrix}
        \mathrm{Prox}_{\sigma g^*_1}(u_1)\\
        \vdots\\
        \mathrm{Prox}_{\sigma g^*_m}(u_m)
    \end{pmatrix}+W\begin{pmatrix}
        u_{m+1}\\
        \vdots\\
        u_{2m}
    \end{pmatrix}.
\end{align*}
Combining the block graph structure of $W$ as described in Definition \ref{defL} (d), the expression for $u$ is
\begin{align*}
    \begin{pmatrix}
        u_1\\
        \vdots\\
        u_m\\
        u_{m+1}\\
        \vdots\\
        u_{2m}
    \end{pmatrix}&= \begin{pmatrix}
        \sigma x_1^t-\lambda_1\\
        \vdots\\
        \sigma x_m^t-\lambda_m\\
        \sigma\sum_{l \in \mathcal{N}_1} L_{1l}x_l^t-\lambda_{m+1}\\
        \vdots\\
        \sigma\sum_{l \in \mathcal{N}_m} L_{ml}x_l^t-\lambda_{2m}
    \end{pmatrix}.
\end{align*}
Since \(H\in\partial\mathrm{Prox}_{\sigma G^*}(u)\), we obtain
\begin{align*}
    H=\begin{pmatrix}
 H_1 &  &  &\\
 & \ddots &  &\\
 &  & H_m&\\
 &&& I_{mn}
\end{pmatrix},
\end{align*}
where $H_i\in \partial\mathrm{Prox}_{\sigma g_i^*}(u_i)$. Next, we present the expanded form of $M\tilde{{d}}^{j}$ as below. 
\begin{align*}
   M \tilde{{d}}^{j}&=V\tilde{{d}}^{j}+\sigma B^THB\tilde{{d}}^{j}\\ 
                  &=\begin{pmatrix}
                      V_1\tilde{{d}}^{j}_1\\ \vdots\\  V_m\tilde{{d}}^{j}_m
                  \end{pmatrix}+\sigma\begin{pmatrix}
                      H_1\tilde{{d}}^{j}_1\\ \vdots\\  H_m\tilde{{d}}^{j}_m
                  \end{pmatrix}+\sigma W^2\tilde{{d}}^{j}.
\end{align*}
Considering the structure of the matrix $W$, we get
\begin{align*}
   W^2\tilde{{d}}^{j}&=W\begin{pmatrix}
        \hat{d}^j_1\\ \vdots\\ \hat{d}^j_m
    \end{pmatrix}=\begin{pmatrix}
                    \sum_{t \in \mathcal{N}_1} L_{1t}\hat{{d}}^{j}_t\\
        \vdots\\
        \sum_{t \in \mathcal{N}_m} L_{mt}\hat{{d}}^{j}_t 
                 \end{pmatrix},
\end{align*}
where
\begin{align*}
    \begin{pmatrix}
        \hat{d}^j_1\\ \vdots\\ \hat{d}^j_m
    \end{pmatrix}&=W\tilde{{d}}^{j}=\begin{pmatrix}
                    \sum_{t \in \mathcal{N}_1} L_{1t}\tilde{{d}}^{j}_t\\
        \vdots\\
        \sum_{t \in \mathcal{N}_m} L_{mt}\tilde{{d}}^{j}_t 
        \end{pmatrix}.  
\end{align*}
Thus, we obtain 
\begin{align*}
    d_i^{j+1} &= \tilde{{d}}_{i}^{j} - \frac{1}{L} \Biggl[ V_i\tilde{d}_i^{j} + \sigma H_i \tilde{d}_i^{j}+ \sigma \sum_{k \in \mathcal{N}_i} L_{ik} \hat{d}_k^{j}  \\
    & +\nabla f_i(x_i^t) + T_i +\sum_{k \in \mathcal{N}_i} L_{ik} u_{k+m} \Biggr],
    \end{align*} 
where ${V}_{i} \in \partial \nabla f_i(x_i^t)$, $T_i = \mathrm{Prox}_{\sigma g_i^*}(u_i)$, $H_i \in \partial \mathrm{Prox}_{\sigma g_i^*}(u_i)$ and $\hat{d}_i^{j} = \sum_{k \in \mathcal{N}_i} L_{ik} \tilde{d}_k^{j}$.
\end{proof}
 From the above analysis, both $V\in\partial(\nabla F)(x^t)$ and $H\in\partial{\mathrm{Prox}}_{\sigma G^*}(\sigma Bx^t-\lambda)$ have the block diagonal structure. In addition, $W$ has the special structure. Thus, the computation in the APG method can be efficiently implemented in a distributed manner. Now we present the DiSSN method at the $i$-th agent as follows:

\begin{algorithm}[H]
\caption{DiSSN method for problem \eqref{eq:minphi} at the $i$-th agent} \label{alg:SSN}
Given $\beta = \frac{\sqrt{L} - \sqrt{\mu}}{\sqrt{L} + \sqrt{\mu}}$ and a sequence $\{\eta_{t}\}$ with $\eta_{t} \to 0$ as $t \to +\infty$. Initialize ${x}_{i}^{0} \in \mathbb{R}^{n}$. For $t=0,1,...$, iterate:
\begin{description}
    \item [Step 1.]\ Compute 
    \begin{align*}
    u_i^t &= \sigma x_i^t - \lambda_i, \\
    u_{i+m}^t &= \sigma \sum_{k \in \mathcal{N}_i} L_{ik} x_k^t - \lambda_{i+m}, 
    \end{align*} 
    and share ${u}_{i+m}^t$ with neighbors in $\mathcal{N}_{i}$.
    \item [Step 2.]\ Choose ${V}_{i}^t \in \partial \nabla f_i(x_i^t)$, $T_i^t = \mathrm{Prox}_{\sigma g_i^*}(u_i^t)$, $H_i^t \in \partial \mathrm{Prox}_{\sigma g_i^*}(u_i^t)$.
    \item [Step 3.]\ Set  $d^{0}_i = d^{-1}_i = \mathbf{0} \in \mathbb{R}^{n}$ and $N_t=N(\eta_t)$. For $j = 0, 1, \ldots, N_t$, iterate:
    \begin{enumerate}
        \item  Compute 
        \begin{equation*}
            \tilde{{d}}_{i}^{j} = {d}_{i}^{j} + \beta \left( {d}_{i}^{j} - {d}_{i}^{j-1} \right)
        \end{equation*} 
        and share it with neighbors in $\mathcal{N}_{i}$. 
        \item  Compute 
        \begin{equation*}
        \hat{d}_i^{j} = \sum_{k \in \mathcal{N}_i} L_{ik} \tilde{d}_k^{j} 
        \end{equation*}
        and share it with neighbors in $\mathcal{N}_{i}$.
        \item  Compute 
         \begin{align*}
    d_i^{j+1} &= \tilde{{d}}_{i}^{j} - \frac{1}{L} \Biggl[ V_i^t\tilde{d}_i^{j} + \sigma H_i^t \tilde{d}_i^{j}+ \sigma \sum_{k \in \mathcal{N}_i} L_{ik} \hat{d}_k^{j} \\
    & +\nabla f_i(x_i^t) + T_i^t+\sum_{k \in \mathcal{N}_i} L_{ik} u_{k+m}^t \Biggr].
    \end{align*} 
    \end{enumerate}

    % \item [Step 1.] Select $M^t\in\hat{\partial}^2 \phi(x^t)$. Solve the following linear system by Algorithm \ref{alg:distributed APG} 
    % \begin{equation}
    %   M^t d = -\nabla \phi(x^t)
    % \end{equation}
    % to find $d^{t}$ such that $\| M^t {d}^{t}+ \nabla \phi(x^t) \| \leq {\eta_t}\|\nabla \phi(x^t)\|$.
    \item [Step 4.]\ Update ${x}_{i}^{t+1} = {x}_{i}^{t} + {d}_{i}^{N_t+1}$.
    \item [Step 5.]\ If the desired stopping criterion is satisfied, terminate; otherwise, set $t:=t+1$ and go to Step 1.
\end{description}
\end{algorithm}

Based on Theorem 7.5.5 in \cite{facchinei2003finite}, we present the convergence result of Algorithm \ref{alg:SSN}.

\begin{theorem}\label{DINN convergence}
  Suppose Assumptions \ref{assump1} and \ref{assump3} hold. Let \(\{\eta_t\}\) be a sequence such that \(\eta_t \to 0\) as \(t \to \infty\). There
 exists a neighborhood $\mathbb{B}(\hat{x},\delta)$, where $\hat{x}$ is the solution of the subproblem \eqref{eq:minphi}, such that for any initial point $x^0\in \mathbb{B}(\hat{x},\delta)$, the sequence $\{x^t\}$ generated by Algorithm \ref{alg:SSN} superlinearly converges to $\hat{x}$. In addition, if for some $\tilde{\eta}$, $\eta_t\leq\tilde{\eta}\|\nabla\phi(x^t)\|$ for all $t$, then the sequence $\{x^t\}$ quadratically converges to $\hat{x}$.
\end{theorem}
\begin{proof}
By Assumption \ref{assump1},  Since $f_i$ and $\mathrm{Prox}_{g_i^*}$ ($i=1,\cdots,m$) are strongly semismooth at the solution $\hat{x}$, by Proposition 7.5.18 of \cite{facchinei2003finite} we have $\nabla \phi$ is strongly semismooth at the solution $\hat{x}$. Based on Theorem 7.5.5 in \cite{facchinei2003finite}, we have the conclusion.   
\end{proof}

\subsection{Initialization for DiSSN method}
\label{sec:Initialization for Distributed Inexact Newton Method }

Although the DiSSN method is an efficient algorithm, it converges locally. To attain global convergence, the backtracking line search technique is often combined in the centralized Newton method. However, in the distributed algorithm, all agents must cooperate to compute the global objective function value, which means that the backtracking line search could lead to excessive communication between agents, potentially impacting the overall efficiency of the algorithm. To mitigate this issue, we may adopt the DAPG method to generate an initial point in $\mathbb {B}(\hat{x},\delta)$ for Algorithm \ref{alg:SSN}.

\begin{algorithm}                     
\caption{DAPG method at the $i$-th agent}  
\label{alg:DDFGM}   
Initialize $x_i^0 = x_i^{-1} = \mathbf{0} \in \mathbb{R}^n$. Set the parameter $\beta = \frac{\sqrt{L} - \sqrt{\mu}}{\sqrt{L} + \sqrt{\mu}}$.  For $j = 0, 1, \ldots$, iterate:
\begin{description}
    \item [Step 1.]\, Compute 
        \begin{equation*}
        \tilde{x}_i^j = x_i^j + \beta(x_i^j - x_{i-1}^j), 
        \end{equation*}
        share it with neighbors in $\mathcal{N}_i$.
    \item [Step 2.]\,  Compute 
        \begin{align*}
        u_i^j &= \sigma x_i - \lambda_i, \\
        u_{i+m}^j &= \sigma \sum_{t \in \mathcal{N}_i} L_{it} \tilde{x}_t^j - \lambda_{i+m},  
        \end{align*}
        and share $u_{i+m}^j$ with neighbors in $\mathcal{N}_{i}$.
    \item [Step 3.]\, Compute $T_i^j \in \mathrm{Prox}_{\sigma g_i^*}(u_i^j).$ 
    \item [Step 4.]\,  Compute 
        \begin{align*}
        x_i^{j+1} &= \tilde{x}_i^j - \frac{1}{L} \left[ \nabla f_i(\tilde{x}_i^j) + T_i^j + \sum_{t \in \mathcal{N}_i} L_{it} u_{t+m}^j \right]. 
        \end{align*}
    \item [Step 5.]\, If the desired stopping criterion is satisfied, terminate; otherwise, set $j:=j+1$ and go to Step 1.
\end{description}
\end{algorithm}

To ensure the convergence of Algorithm \ref{alg:SSN}, it is important and meaningful to estimate the iteration complexity of Algorithm \ref{alg:DDFGM}. Now we present the following theorem about the estimation.
\begin{theorem}
If $\{x^j\}, j=0,1,\cdots,$ is a sequence generated by Algorithm \ref{alg:DDFGM}, then for any $j$, we have
\begin{equation}
\label{eq:iteration-complexity}
    \|x^j-\hat{x}\|^2\leq\frac{L+\mu}{\mu}\|x^0-\hat{x}\|^2e^{-j\sqrt{\frac{\mu}{L}}}.
\end{equation} 
Furthermore, if $j>N:=\lceil {\sqrt{\frac{L}{\mu}}\mathrm{ln}(\frac{L+\mu}{\mu\delta^2}\|x^0-\hat{x}\|^2)}\rceil$, $x^j$ enters the convergence neighborhood with the radius $\delta$.
\end{theorem}
\begin{proof} 
From Theorem 2.1.5, Theorem 2.2.1 and Lemma 2.2.4 in \cite{Nesterov2018}, we can obtain
\begin{equation}
    \phi(x^j)-\phi(\hat{x})\leq\frac{L+\mu}{2}\|x^0-\hat{x}\|^2e^{-j\sqrt{\frac{\mu}{L}}}.\label{pf1}
\end{equation} 
Since the function \(\phi({x})\) is \(\mu\)-strongly convex and differentiable, we have
\begin{equation}
    \frac{\mu}{2}\|x^j-\hat{x}\|^2\leq \phi(x^j)-\phi(\hat{x}).\label{pf2}
\end{equation} 

Combining \eqref{pf1} with \eqref{pf2}, we get \eqref{eq:iteration-complexity} and the estimation of the iteration number.
\end{proof} 

\begin{remark}
Although we have given the iteration complexity in theory, the parameters $\hat{x}$ and $\delta$ are unknown. Therefore, in practical computation, we may set an appropriate accuracy and maximum iteration number for Algorithm \ref{alg:DDFGM} according to practical problems.   
\end{remark}

\subsection{Restatement of ALM}
\label{subsec:restatement of ALM}
While we have previously introduced ALM, %in Algorithm \ref{alg:SSNAL},
it does not reflect the strategy of distributed computation. Now, by combining Algorithm %\ref{alg:SSNAL}, \ref{alg:distributed APG} and 
\ref{alg:SSN}, we present ALM with the distributed computation details, called the DSSNAL method at the $i$-th agent.

\begin{algorithm}
\caption{DSSNAL method for problem \eqref{eq:problemBx=y} at the $i$-th agent} \label{alg:DSSNAL}
Given  $\sigma_0 > 0$. Initialize ${x}_{i}^{0} \in \mathbb{R}^{n}$, $\lambda_{i}^{0} \in \mathbb{R}^{n}$ and $\lambda_{i+m}^{0} \in \mathbb{R}^{n}$. For $k=0,1,\cdots$, iterate:
\begin{description}
    \item [Step 1.] \ Compute $x_i^{k+1}$ by Algorithm \ref{alg:SSN}.
    \item [Step 2.]\ Compute 
    \begin{align*}
    u_i^{k+1} &= \sigma_{k+1} x_i^{k+1} - \lambda_i^k, \\
    u_{i+m}^{k+1} &= \sigma_{k+1} \sum_{t \in \mathcal{N}_i} L_{it} {x}_t^{k+1} - \lambda_{i+m}^k,  
    \end{align*}
share  $u_{i+m}^{k+1}$ with neighbors in $\mathcal{N}_i$.
    \item [Step 3.] \ Compute 
    \begin{align*}
    \lambda^{k+1}_i &=-\mathrm{Prox}_{\sigma_{k+1}g^*_i}(u_i^{k+1}),\\
    \lambda^{k+1}_{i+m} &=-\sum_{t \in \mathcal{N}_i} L_{it} u^{k+1}_{t+m}.
    \end{align*}

    \item [Step 4.]\ Update $\sigma_{k+1}$, satisfying $0 \leq \sigma_k\uparrow\sigma_\infty \leq \infty$. If the desired stopping criterion is satisfied, terminate; otherwise, set $k:=k+1$ and go to Step 1.
\end{description}
\end{algorithm}

\section{Numerical experiments}
\label{sec:Numerical experiments}

In this section, we conduct numerical experiments to assess the performances of two algorithms. We compare the DSSNAL method with the FDPG method \cite{chen2012fast} and the $\mathrm{Prox}$-NIDS method \cite{JinmingDistributed}. The algorithms are tested both on the random data and the real data from the UCI-data set \cite{mj07-xa20-26}.  All experiments are implemented in {\sc Matlab} R2018b on a PC with Intel Core i5-10200H processors (2.40GHz) and 16 GB of RAM.

We terminate the algorithms when
\begin{equation*}
    R_{{\rm KKT}}:=\frac{\|Wx\|+\|x-\mathrm{Prox}_{\widetilde{G}}(x-A\nabla F(x))\|}{1+\|x\|}<10^{-6},
\end{equation*}
where $A=\frac{1}{m} \left( \mathbf{1}_{m \times m} \otimes I_n \right)$ and $\widetilde{G}(y):=\sum\limits_{i=1}^{m}g_i(y_i)$ with $y=[y_1^T,\dots,y_m^T]^T\in \mathbb{R}^{mn}$. The algorithm also terminates if the number of iterations reaches the maximum $N_{\mathrm{max}}$.
 For the DSSNAL method, we set $N_{\mathrm{max}}$ to be $100$; for the FDPG method, we set $N_{\mathrm{max}}$ to be 300000; for the $\mathrm{Prox}$-NIDS method, we set $N_{\mathrm{max}}$ to be 60000. In addition,
for the DSSNAL method, we call Algorithm \ref{alg:DDFGM} to obtain an initial point. We terminate  Algorithm \ref{alg:DDFGM} when
\begin{equation*}
    \frac{\|\nabla\phi(x)\|}{1+\|x\|}\leq 5\times10^{-1}.
\end{equation*}

%Besides, the network consists of \(m=50\) agents. Each edge is uniformly and randomly assigned to two agents while maintaining network connectivity. 

%Finally, in all numerical experiments, we standardize the dataset using Z-score normalization. Also referred to as standardization, this method transforms the data to have a mean of 0 and a standard deviation of 1.

\subsection{Huber regression problem}\label{Huber}

In this subsection, we compare the algorithms on the Huber regression problem 
\begin{equation*}
\label{eq:Huber regression}
  \min_{w\in\mathbb{R}^n}\sum\limits_{i=1}^{S}\Big\{f_{i}(w)+g_{i}(w)\Big\}  
\end{equation*}
with
\begin{align*}
f_i(w) &= \sum\limits_{j\in J_i}l_h^{\nu}(a_j^Tw-b_j)+\frac{\rho}{2m}\|w\|^2, \\ 
g_i(w) &= \frac{1}{m}\gamma\|w\|_1,
\end{align*}
where $S$ denotes the total number of functions $f_i(w)$ and $g_i(w)$, \(\{a_j\}_{j=1}^S\subseteq  \mathbb{R}^n\), \(\{b_j\}_{j=1}^S\subseteq \mathbb{R}\), $\rho, \gamma>0$ are given parameters, $m$ denotes the total number of agents, \(J_i\) collects the indices of the data points assigned to the $i$-th agent $(i=1\dots m)$, and
\begin{equation*}
   l_h^{\nu}(t) =  
\begin{cases} 
\frac{1}{2\nu} t^2, & \text{if } |t| \leq \nu, \\|t| - \frac{\nu}{2}, & \text{otherwise.}
\end{cases} 
\end{equation*}
% Furthermore, we assign the data uniformly and randomly to each agent. 

 Note that \(f_i\) is strongly convex and continuously differentiable and \(g_i\) is convex with
\begin{align*}
    \nabla f_i(w)&=\frac{1}{\nu}\sum\limits_{j \in J_i}\mathcal{T}_{\nu}(a_j^Tw-b_j)a_j+\frac{\rho}{m}w
\end{align*}
and
\begin{align*}
       \mathrm{Prox}_{\sigma g_i^*}(w)&=(\mathcal{T}_{\frac{\gamma}{m}}(w_1),\dots,\mathcal{T}_{\frac{\gamma}{m}}(w_n))^T,
\end{align*}
 where \(\mathcal{T}_{\nu}: \mathbb{R}\to  \mathbb{R}\) is defined as \(\mathcal{T}_{\nu}(t):=\mathrm{sgn}(t)\mathrm{min}\{|t|,\nu\}\). Additionally, the Clarke subdifferentials of \(\nabla f_i\) and \(\mathrm{Prox}_{\sigma g_i^*}\) are
\begin{align*}
  \partial\nabla f_i(w)&=\frac{1}{\nu}\sum\limits_{j \in J_i}a_ja_j^T\mathcal{D}^{\nu}(a_j^Tw-b_j)+\frac{\rho}{m}I_n
  \end{align*}
  and
  \begin{align*}
  \partial \mathrm{Prox}_{\sigma g_i^*(w)}&=\begin{pmatrix}
         \mathcal{D}^{\frac{\gamma}{m}}(w_1)&&&\\&&\ddots\\ &&&\mathcal{D}^{\frac{\gamma}{m}}(w_1)\end{pmatrix}, 
\end{align*}
where \(\mathcal{D}^{\nu}: \mathbb{R}\to  \mathbb{R}\) is defined as

 \begin{align}
    \mathcal{D}^{\nu}(t):= \begin{cases}       
    1, & \text{if } |t| < \nu, \\0, & \text{if } |t| > \nu, \\ [0,1], & \text{if } |t| = \nu.\label{eq:D} 
\end{cases} 
 \end{align}

\subfile{table-compare-table1.dat}

\subfile{table-compare-table2.dat}

\subfile{table-compare-table3.dat}

\subfile{table-compare-table4.dat}

In practical computation, we set $\nu=1$, $\rho=1$, and $m=50$. The parameter $\gamma$ is obtained via the five-fold cross validation. In the numerical results, we report the data set name (dataname), the relative KKT residual ($R_{\rm KKT}$), the computing time (time), the iteration number (iter), the value of the objective function (obj), the number of agents ($m$), the dimension of the variables ($n$), the number of functions ($S$). We present a number `$s\times 10^t$' in the format of `$s\  \mathrm{sign}(t)|t|$', e.g., $1.0\text{-}4$ denotes $1.0\times10^{-4}$. The computing time is in the format of `hours:minutes:seconds'. For the DSSNAL method, we present the iteration number in the format of $s_1(s_2)$, where $s_1$ denotes the iteration number of the DSSNAL method and $s_2$ denotes the iteration number of the DAPG method. Here, the index $s_2$ is very important since it is closely related to the number of communications.

For the random data problems, the matrix $A:=(a_1,a_2,\dots,a_S)\in\mathbb{R}^{n\times S}$  and the vector $b:=(b_1,b_2,\dots,b_S)^T\in\mathbb{R}^S$ are generated by the following code:
\begin{verbatim}
b = rand(S,1);  A = randn(n,S);
min_val = min(A,[],1); 
max_val = max(A,[],1);
val_ranges = max(max_val-min_val, 1e-10);
A = (A-min_val)./val_ranges;
\end{verbatim}
The matrix \( L \) defined in Definition \ref{defL} is generated by the following code:
\begin{verbatim}
 e = ones(m,1); Q = null(e'); L = Q*Q';
\end{verbatim}
We present the results in Table I. It is evident that the DSSNAL method not only attains the desired accuracy, but also demonstrates its superiority in efficiency compared to the FDPG method and the $\mathrm{Prox}$-NIDS method. For example, for the problem `rand(20,4000)', the DSSNAL method takes less than 2 minutes to achieve the required accuracy. In contrast, the $\mathrm{Prox}$-NIDS method  and the FDPG method fail.

For the UCI-real data problems, we standardize the dataset using the Z-score normalization, which transforms the data to have a mean of 0 and a standard deviation of 1. For the real data problems, we present the results in Table II. From the table, we can see that among the three algorithms, only the DSSNAL method attains the desired accuracy for all the test problems. Meanwhile, the DSSNAL method takes much less time than the other two algorithms.

\subsection{Support vector classification problem} \label{subsec:svc}

In this subsection, we compare the algorithms on the support classification problem
\begin{equation*}
\label{eq:svc}
  \min_{w\in\mathbb{R}^n}\sum\limits_{i=1}^{S}\Big\{f_{i}(w)+g_{i}(w)\Big\}  
\end{equation*}
with
\begin{align*}
    f_i(w)&=C\sum\limits_{j\in J_i}\left( \mathrm{max} \big( 0,1 - b_{j} a_{j}^{T} {w} \big) \right)^{2}+\frac{\rho}{2m}\|w\|^2,\\
      g_i(w)&=\frac{1}{m}\gamma\|w\|_1,
\end{align*}
where \(\{a_j\}_{j=1}^S\subseteq  \mathbb{R}^n\), \(\{b_j\}_{j=1}^S\subseteq \{-1,+1\}\), \(C, \rho>0\) are given parameters, \(J_i\) collects the indices of the data points assigned to the $i$-th agent. Similarly, \(f_i\) is strongly convex and continuously differentiable and \(g_i\) is convex with
\begin{align*}
    \nabla f_i(w)&=-2C\sum\limits_{j\in J_i}\left( \mathrm{max} \big( 0,1 - b_{j} a_{j}^{T} {w} \big) \right)b_ja_j+\frac{\rho}{m}w
\end{align*}
and
\begin{align*} 
       \mathrm{Prox}_{\sigma g_i^*(w)}&=(\mathcal{T}_{\frac{\gamma}{m}}(w_1),\dots,\mathcal{T}_{\frac{\gamma}{m}}(w_n))^T,
\end{align*}where \(\mathcal{T}_{\nu}(t)=\mathrm{sgn}(t)\min\{|t|,\nu\}\). Additionally, the Clarke subdifferentials of \(\nabla f_i\) and \(\mathrm{Prox}_{\sigma g_i^*}\) are 
\begin{align*}
     \partial\nabla f_i(w)&=2C \sum_{j \in {J}_i} {a}_j {a}_j^T \mathcal{P}(1 - b_j {a}_j^T{w})+\frac{\rho}{m}I_n 
\end{align*}
and
\begin{align*} 
       \partial \mathrm{Prox}_{\sigma g_i^*(w)}&=\begin{pmatrix} \mathcal{D}^{\frac{\gamma}{m}}(w_1)&&&\\&&\ddots\\ &&&\mathcal{D}^{\frac{\gamma}{m}}(w_1)\end{pmatrix},
\end{align*}
where \(\mathcal{D}^{\nu}\) is defined in \eqref{eq:D} and the multifunction \(\mathcal{P}: \mathbb{R}\to  \mathbb{R}\) is the Clarke subdifferential of \(\max(0,t)\), i.e., 
 \begin{equation} 
    \mathcal{P}(t) = 
    \begin{cases} 1, & \text{if } t > 0; \\0, & \text{if } t < 0; \\ [0,1], & \text{if } t = 0.
    \end{cases}
\end{equation} 

In practical computation, we set $\rho=1$, $m=50$. The parameter $\gamma$ and $C$ are obtained via the five-fold cross validation.

Tables III and IV present the performance comparisons of the three algorithms on a selection of  support vector classification problems. Table III is for the random data problems and Table IV is for the UCI-data problems. The results indicate that  similar to the previous subsection, the DSSNAL method requires much less time to meet the termination criterion  compared to the other algorithms. Additionally, we observe that in some datasets, the other algorithms fail to achieve the desired accuracy.

\section{Conclusion}
\label{sec:conclusion}
This paper has presented a DSSNAL method for solving decentralized optimization problems over networks. By reformulating the original problem with local variables and consensus constraints, and solving the resulting subproblems via the DiSSN method, the proposed approach efficiently balances computation and communication. The employment of the DAPG method further avoids full Hessian communication, enhancing scalability. Theoretical convergence guarantees have been established, and numerical experiments demonstrate the algorithm's superior performance compared to existing distributed algorithms.
\bibliographystyle{IEEEtran}
\medskip
\bibliography{reference}
\end{document}

%% file: table-compare-table1.dat
\begin{table*}[!thbp]
%\footnotesize
\centering
\parbox{0.8\textwidth}{\caption{Numerical comparisons for the Huber regression  problem on the random data problem $(m=50)$. In this table `$a$'=FDPG, `$b$'=DSSNAL, `$c$'=Prox-NIDS.}\label{table-compare-random-convex-all}}
%\resizebox{1.0\textwidth}{40mm}{
\begin{tabular}{|c|c|c|c|c|c|}
\hline
\multirow{2}{*}{dataname(n,S)}&\multirow{2}{*}{$\gamma$}&\multicolumn{1}{c|}{iter}&\multicolumn{1}{c|}{$R_{{\rm KKT}}$}&\multicolumn{1}{c|}{obj}&\multicolumn{1}{c|}{time}\\\cline{3-6} & &  $a\ |\ b\ |\ c $ & $a\ |\ b\ |\ c$ & $a\ |\ b\ |\ c$& $a\ |\ b\ |\ c$ \\[2pt] 
\hline

\multirow{2}{*}{rand(2,400)} 
	&\multirow{2}{*}{ 4.2-5  }&\multirow{2}{*}{ 300000 $|$ 5(26247)  $|$ 5862 } &\multirow{2}{*}{ 3.07-1 $|$ 9.21-9   $|$ 9.98-7     } &\multirow{2}{*}{ 6.218+1 $|$ 6.170+1  $|$ 6.170+1  }&\multirow{2}{*}{ 11:08 $|$ 0:02  $|$ 0:06 }  
\\[2pt] 
& & & & & \\[2pt] 
\hline
\multirow{2}{*}{rand(4,800)} 
	&\multirow{2}{*}{ 2.7-4  }&\multirow{2}{*}{ 300000 $|$ 5(82125)  $|$ 29587 } &\multirow{2}{*}{ 2.96-1 $|$ 4.32-7  $|$ 9.89-7      } &\multirow{2}{*}{ 1.432+2 $|$ 1.312+2  $|$ 1.312+2 }&\multirow{2}{*}{ 11:48 $|$ 0:03  $|$ 1:10 }  
\\[2pt] 
& & & & & \\[2pt] 
\hline
\multirow{2}{*}{rand(6,1200)} 
	&\multirow{2}{*}{ 5.2-4  }&\multirow{2}{*}{ 300000 $|$ 6(138084) $|$ 60000  } &\multirow{2}{*}{ 3.02-1 $|$ 1.34-8   $|$ 1.10-5     } &\multirow{2}{*}{ 2.150+2 $|$ 1.949+2  $|$ 1.949+2  }&\multirow{2}{*}{ 12:05 $|$ 0:09  $|$ 3:07 }  
\\[2pt] 
& & & & & \\[2pt] 
\hline
\multirow{2}{*}{rand(8,1600)} 
	&\multirow{2}{*}{ 6.3-4  }&\multirow{2}{*}{ 300000 $|$ 6(74454)  $|$ 60000 } &\multirow{2}{*}{ 2.91-1 $|$ 6.44-8  $|$ 1.21-3      } &\multirow{2}{*}{ 2.781+2 $|$ 2.659+2 $|$ 2.668+2  }&\multirow{2}{*}{ 14:26 $|$ 0:13  $|$ 5:02 }  
\\[2pt] 
& & & & & \\[2pt] 
\hline
\multirow{2}{*}{rand(10,2000)} 
	&\multirow{2}{*}{ 8.8-4  }&\multirow{2}{*}{ 300000 $|$ 6(123938)  $|$ 60000 } &\multirow{2}{*}{ 3.23-1 $|$ 1.67-7  $|$ 8.73-3      } &\multirow{2}{*}{ 3.473+2 $|$ 3.309+2  $|$ 3.312+2  }&\multirow{2}{*}{ 15:49 $|$ 0:23  $|$ 6:26 }  
\\[2pt] 
& & & & & \\[2pt] 
\hline
\multirow{2}{*}{rand(12,2400)} 
	&\multirow{2}{*}{ 7.3-4  }&\multirow{2}{*}{ 300000 $|$ 6(53921)  $|$ 60000 } &\multirow{2}{*}{ 2.52-1 $|$ 4.21-7   $|$ 1.44-2     } &\multirow{2}{*}{ 4.090+2 $|$ 3.996+2  $|$ 4.002+2  }&\multirow{2}{*}{ 17:28 $|$ 0:27  $|$ 7:19 }  
\\[2pt] 
& & & & & \\[2pt] 
\hline
\multirow{2}{*}{rand(14,2800)} 
	&\multirow{2}{*}{ 8.4-4  }&\multirow{2}{*}{ 300000 $|$ 7(78527)   $|$ 60000} &\multirow{2}{*}{ 2.83-1 $|$ 8.45-7  $|$ 1.48-2      } &\multirow{2}{*}{ 4.851+2 $|$ 4.696+2   $|$  4.701+2 }&\multirow{2}{*}{ 20:47 $|$ 0:31  $|$ 8:06 }  
\\[2pt] 
& & & & & \\[2pt] 
\hline
\multirow{2}{*}{rand(16,3200)} 
	&\multirow{2}{*}{ 8.6-4  }&\multirow{2}{*}{ 300000 $|$ 7(194682)  $|$ 60000 } &\multirow{2}{*}{ 2.97-1 $|$ 1.30-8  $|$ 1.69-2      } &\multirow{2}{*}{ 5.629+2 $|$ 5.368+2  $|$ 5.373+2  }&\multirow{2}{*}{ 24:36 $|$ 0:52  $|$ 9:57 }  
\\[2pt] 
& & & & & \\[2pt] 
\hline
\multirow{2}{*}{rand(18,3600)} 
	&\multirow{2}{*}{ 1.2-3  }&\multirow{2}{*}{ 300000 $|$ 7(287982)  $|$ 60000 } &\multirow{2}{*}{ 2.12-1 $|$ 2.42-8    $|$ 1.85-2    } &\multirow{2}{*}{ 6.427+2 $|$ 5.977+2  $|$ 5.989+2  }&\multirow{2}{*}{ 28:32 $|$ 1:25  $|$ 11:36}  
\\[2pt] 
& & & & & \\[2pt] 
\hline
\multirow{2}{*}{rand(20,4000)} 
	&\multirow{2}{*}{ 1.8-3  }&\multirow{2}{*}{ 300000 $|$ 7(314475)  $|$ 60000 } &\multirow{2}{*}{ 2.65-1 $|$ 3.33-8  $|$ 1.71-2      } &\multirow{2}{*}{ 7.116+2 $|$ 6.674+2  $|$ 6.696+2 }&\multirow{2}{*}{ 30:26 $|$ 1:38  $|$ 13:51 }  
\\[2pt] 
& & & & & \\[2pt] 
\hline

\end{tabular}
%}
\end{table*}

%% file: table-compare-table2.dat
\begin{table*}[!thbp]
%\footnotesize
\centering
\parbox{0.8\textwidth}{\caption{Numerical comparisons for the Huber regression  problem  on the UCL datasets$(m=50)$. In this table `$a$'=FDPG, `$b$'=DSSNAL, `$c$'=Prox-NIDS.}\label{table-compare-random-convex-all}}
%\resizebox{1.0\textwidth}{40mm}{
\begin{tabular}{|c|c|c|c|c|c|}
\hline
\multirow{2}{*}{dataname(n,S)}&\multirow{2}{*}{$\gamma$}&\multicolumn{1}{c|}{iter}&\multicolumn{1}{c|}{$R_{{\rm KKT}}$}&\multicolumn{1}{c|}{obj}&\multicolumn{1}{c|}{time}\\\cline{3-6} & &  $a\ |\ b\ |\ c $ & $a\ |\ b\ |\ c$ & $a\ |\ b\ |\ c$& $a\ |\ b\ |\ c$ \\[2pt] 
\hline

\multirow{2}{*}{abalone.scale(8,4177)} &\multirow{2}{*}{2.9-2}
	&\multirow{2}{*}{ 300000 $|$ 7(221676) $|$ 60000 } &\multirow{2}{*}{ 1.84-2 $|$  5.05-7$|$ 2.78-3} &\multirow{2}{*}{ 9.724+2 $|$ 9.688+2$|$ 9.689+2}&\multirow{2}{*}{  10:31 $|$ 0:32  $|$ 10:03 }  
\\[2pt] 
 & & & & &\\[2pt] 
\hline

\multirow{2}{*}{ailerons51tst(40,2750)} &\multirow{2}{*}{1.1-2}
	&\multirow{2}{*}{ 300000 $|$ 8(482837) $|$ 60000 } &\multirow{2}{*}{ 9.28-3 $|$ 3.21-7$|$ 2.68-6} &\multirow{2}{*}{ 9.972+2 $|$ 9.948+2 $|$ 9.948+2}&\multirow{2}{*}{  48:23 $|$ 12:14 $|$ 20:16  }  
\\[2pt] 
& & & & &\\[2pt] 
\hline

\multirow{2}{*}{anacalt(7,4052)} &\multirow{2}{*}{3.5-6}
	&\multirow{2}{*}{ 300000 $|$ 8(172518) $|$ 60000 } &\multirow{2}{*}{ 1.83-2 $|$ 1.13-7 $|$ 9.84-3 } &\multirow{2}{*}{ 1.289+3 $|$ 1.287+3$|$ 1.287+3}&\multirow{2}{*}{  11:19 $|$ 0:42 $|$ 8:08  }  
\\[2pt] 
 & & & & & \\[2pt] 
\hline

\multirow{2}{*}{concrete(8,1030)} &\multirow{2}{*}{8.0-4}
	&\multirow{2}{*}{ 300000 $|$ 6(188681) $|$ 60000 } &\multirow{2}{*}{ 4.21-2 $|$ 3.57-7$|$ 7.65-4 } &\multirow{2}{*}{ 4.165+2 $|$ 4.088+2$|$ 4.089+2}&\multirow{2}{*}{  12:47 $|$ 0:09 $|$ 3:11  }  
\\[2pt] 
& & & & &\\[2pt] 
\hline

\multirow{2}{*}{friedman(5,1200)} &\multirow{2}{*}{5.0-3}
	&\multirow{2}{*}{ 300000 $|$ 5(84622) $|$ 53798 } &\multirow{2}{*}{ 2.92-1 $|$ 3.72-7$|$ 9.86-7 } &\multirow{2}{*}{ 4.613+2 $|$ 4.400+2 $|$ 4.400+2}&\multirow{2}{*}{   11:13 $|$  0:12$|$ 2:49  }  
\\[2pt] 
& & & & &\\[2pt] 
\hline

\multirow{2}{*}{puma32h51tst(32,1639)} &\multirow{2}{*}{1.7-3}
	&\multirow{2}{*}{ 300000 $|$ 7(383865)$|$ 60000  } &\multirow{2}{*}{ 2.25-2 $|$ 8.58-7 $|$ 1.78-3} &\multirow{2}{*}{ 7.506+2 $|$ 7.474+2$|$ 7.491+2}&\multirow{2}{*}{  34:32 $|$ 5:07 $|$ 13:36  }  
\\[2pt] 
 & & & & &\\[2pt] 
\hline

\end{tabular}
%}
\end{table*}

%% file: table-compare-table3.dat
\begin{table*}[!thbp]
%\footnotesize
\centering
\parbox{0.8\textwidth}{\caption{Numerical comparisons for the support vector classification problem  on the random data problem $(m=50)$. In this table `$a$'=FDPG, `$b$'=DSSNAL, `$c$'=Prox-NIDS.}\label{table-compare-random-convex-all}}
%\resizebox{1.0\textwidth}{40mm}{
\begin{tabular}{|c|c|c|c|c|c|}
\hline
\multirow{2}{*}{dataname(n,S)}&\multirow{2}{*}{$\gamma$}&\multicolumn{1}{c|}{iter}&\multicolumn{1}{c|}{$R_{{\rm KKT}}$}&\multicolumn{1}{c|}{obj}&\multicolumn{1}{c|}{time}\\\cline{3-6} & &  $a\ |\ b\ |\ c $ & $a\ |\ b\ |\ c$ & $a\ |\ b\ |\ c$& $a\ |\ b\ |\ c$ \\[2pt] 
\hline

\multirow{2}{*}{rand(6,300)} 
	&\multirow{2}{*}{ 3.7-4  }&\multirow{2}{*}{ 300000 $|$ 6(28834)  $|$ 18193 } &\multirow{2}{*}{ 4.17-1 $|$ 6.05-9  $|$ 9.94-7     } &\multirow{2}{*}{ 3.457+2 $|$ 2.955+2   $|$ 2.955+2 }&\multirow{2}{*}{6:52 $|$ 0:07  $|$ 0:26 }  
\\[2pt] 
 & & & & & \\[2pt] 
\hline
\multirow{2}{*}{rand(7,600)} 
	&\multirow{2}{*}{ 5.4-4  }&\multirow{2}{*}{ 300000 $|$ 6(29018)  $|$ 43939  } &\multirow{2}{*}{ 4.69-1 $|$ 6.21-8  $|$ 9.93-7      } &\multirow{2}{*}{ 6.092+2 $|$ 5.949+2   $|$ 5.949+2  }&\multirow{2}{*}{7:17 $|$ 0:13  $|$ 1:23 }  
\\[2pt] 
 & & & & & \\[2pt] 
\hline
\multirow{2}{*}{rand(8,900)} 
	&\multirow{2}{*}{ 6.8-4  }&\multirow{2}{*}{ 300000 $|$ 6(29254)  $|$ 60000 } &\multirow{2}{*}{ 3.37-1 $|$ 3.29-7    $|$ 1.63-5   } &\multirow{2}{*}{ 9.216+2 $|$ 8.946+2   $|$ 8.946+2 }&\multirow{2}{*}{9:46 $|$ 0:20  $|$ 3:11 }  
\\[2pt] 
 & & & & & \\[2pt] 
\hline
\multirow{2}{*}{rand(9,1200)} 
	&\multirow{2}{*}{ 8.6-4  }&\multirow{2}{*}{ 300000 $|$ 6(29537)  $|$ 60000 } &\multirow{2}{*}{ 3.46-1 $|$ 9.01-7   $|$ 5.11-4    } &\multirow{2}{*}{ 1.349+3 $|$ 1.193+3   $|$ 1.195+3 }&\multirow{2}{*}{11:35 $|$ 0:24 $|$ 3:57  }  
\\[2pt] 
 & & & & & \\[2pt] 
\hline
\multirow{2}{*}{rand(10,1500)} 
	&\multirow{2}{*}{ 1.5-3  }&\multirow{2}{*}{ 300000 $|$ 7(49682)  $|$ 60000 } &\multirow{2}{*}{ 4.27-1 $|$ 2.37-7  $|$ 2.17-3     } &\multirow{2}{*}{ 1.507+3 $|$ 1.484+3   $|$ 1.487+3 }&\multirow{2}{*}{12:57 $|$ 0:38  $|$ 4:55 }  
\\[2pt] 
 & & & & & \\[2pt] 
\hline
\multirow{2}{*}{rand(11,1800)} 
	&\multirow{2}{*}{ 1.7-3  }&\multirow{2}{*}{ 300000 $|$ 7(50129)  $|$ 60000 } &\multirow{2}{*}{ 2.50-1 $|$ 3.60-8   $|$ 8.65-3    } &\multirow{2}{*}{ 2.094+3 $|$ 1.783+3   $|$ 1.783+3 }&\multirow{2}{*}{13:21 $|$ 0:40  $|$ 5:38 }  
\\[2pt] 
 & & & & & \\[2pt] 
\hline
\multirow{2}{*}{rand(12,2100)} 
	&\multirow{2}{*}{ 1.4-3  }&\multirow{2}{*}{ 300000 $|$ 7(52953)  $|$ 60000 } &\multirow{2}{*}{ 4.33-1 $|$ 1.10-7   $|$ 9.72-3    } &\multirow{2}{*}{ 2.126+3 $|$ 2.090+3  $|$ 2.091+3 }&\multirow{2}{*}{14:38 $|$ 0:46  $|$ 6:52 }  
\\[2pt] 
 & & & & & \\[2pt] 
\hline
\multirow{2}{*}{rand(13,2400)} 
	&\multirow{2}{*}{ 1.9-3  }&\multirow{2}{*}{ 300000 $|$ 7(92687)  $|$ 60000 } &\multirow{2}{*}{ 3.41-1 $|$ 3.93-7   $|$ 1.19-2    } &\multirow{2}{*}{ 2.526+3 $|$ 2.384+3  $|$ 2.387+3  }&\multirow{2}{*}{15:44 $|$ 0:55  $|$ 7:44 }  
\\[2pt] 
 & & & & & \\[2pt] 
\hline
\multirow{2}{*}{rand(14,2700)} 
	&\multirow{2}{*}{ 2.1-3  }&\multirow{2}{*}{ 300000 $|$ 7(65036)  $|$ 60000 } &\multirow{2}{*}{ 3.94-1 $|$ 2.34-7   $|$ 9.97-3    } &\multirow{2}{*}{ 2.872+3 $|$ 2.683+3  $|$ 2.686+3  }&\multirow{2}{*}{16:22 $|$ 1:09  $|$ 8:37 }  
\\[2pt] 
 & & & & & \\[2pt] 
\hline
\multirow{2}{*}{rand(15,3000)} 
	&\multirow{2}{*}{ 2.6-3  }&\multirow{2}{*}{ 300000 $|$ 7(118647)  $|$ 60000 } &\multirow{2}{*}{ 4.13-1 $|$ 3.54-7   $|$ 9.65-3    } &\multirow{2}{*}{ 3.251+3 $|$ 2.978+3  $|$ 2.993+3  }&\multirow{2}{*}{18:07 $|$ 1:15  $|$ 9:35 }  
\\[2pt] 
 & & & & & \\[2pt] 
\hline

\end{tabular}
%}
\end{table*}

%% file: table-compare-table4.dat
\begin{table*}[!thbp]
%\footnotesize
\centering
\parbox{0.8\textwidth}{\caption{Numerical comparisons for the support vector classification problem  on the UCL datasets$(m=50)$. In this table `$a$'=FDPG, `$b$'=DSSNAL, `$c$'=Prox-NIDS.}\label{table-compare-random-convex-all}}
%\resizebox{1.0\textwidth}{40mm}{
\begin{tabular}{|c|c|c|c|c|c|}
\hline
\multirow{2}{*}{dataname(n,S)}&\multirow{2}{*}{$\gamma$}&\multicolumn{1}{c|}{iter}&\multicolumn{1}{c|}{$R_{{\rm KKT}}$}&\multicolumn{1}{c|}{obj}&\multicolumn{1}{c|}{time}\\\cline{3-6} & &  $a\ |\ b\ |\ c $ & $a\ |\ b\ |\ c$ & $a\ |\ b\ |\ c$& $a\ |\ b\ |\ c$ \\[2pt] 
\hline

\multirow{2}{*}{abalone.scale(8,4177)} &\multirow{2}{*}{2.9-2}
	&\multirow{2}{*}{ 300000$|$  8(225632) $|$ 60000} &\multirow{2}{*}{ 3.09-2 $|$ 6.01-7 $|$ 3.15-3     } &\multirow{2}{*}{ 2.598+3 $|$ 2.586+3 $|$ 2.587+3  }&\multirow{2}{*}{10:05 $|$ 1:15 $|$ 9:34 }  
\\[2pt] 
& & & & &\\[2pt] 
\hline

\multirow{2}{*}{anacalt(7,4052)} &\multirow{2}{*}{3.2-2}
	&\multirow{2}{*}{ 300000 $|$ 8(171935)$|$ 60000  } &\multirow{2}{*}{ 4.03-2 $|$  7.97-7  $|$3.85-3   } &\multirow{2}{*}{ 3.682+3 $|$  3.624+3  $|$ 3.625+3}&\multirow{2}{*}{ 10:19$|$ 1:12 $|$ 8:46 }  
\\[2pt] 
 & & & & &\\[2pt] 
\hline

\multirow{2}{*}{concrete(8,1030)} &\multirow{2}{*}{2.1-3}
     &\multirow{2}{*}{ 300000 $|$  7(102761)$|$ 60000 } &\multirow{2}{*}
{  5.96-2 $|$   7.84-8 $|$ 6.44-4   } &\multirow{2}{*}{ 5.139+2 $|$  5.108+2 $|$ 5.108+2 }&\multirow{2}{*}{ 10:43$|$ 0:24$|$ 3:38 }
	
\\[2pt] 
 & & & & &\\[2pt] 
\hline

\multirow{2}{*}{friedman(5,1200 )} &\multirow{2}{*}{4.2-3}
	  &\multirow{2}{*}{ 300000 $|$  6(26883) $|$ 54779} &\multirow{2}{*}{ 3.39-1 $|$   1.63-7  $|$ 9.98-7   } &\multirow{2}{*}{ 1.108+3 $|$ 1.057+3 $|$ 1.057+3  }&\multirow{2}{*}{ 5:52$|$ 0:19 $|$ 2:49 }  
\\[2pt] 
& & & & &\\[2pt] 
\hline

\multirow{2}{*}{mortgage(15,1649)} &\multirow{2}{*}{1.5-3}
	&\multirow{2}{*}{  300000 $|$ 6(172642) $|$ 60000  } &\multirow{2}{*}{  8.73-3 $|$ 6.35-7 $|$ 2.51-3      } &\multirow{2}{*}{  1.219+2 $|$ 1.197+2 $|$ 1.217+2   }&\multirow{2}{*}{ 16:23 $|$ 0:39 $|$ 6:42  }  
\\[2pt] 
 & & & & &\\[2pt] 
\hline
\end{tabular}
%}
\end{table*}